\documentclass[reqno]{amsart}

\usepackage{graphicx,pdfsync,amssymb,mathrsfs,amsfonts,amsbsy,color,esint,mathtools,tikz,mdframed}
\usetikzlibrary{positioning}
\usepackage[shortlabels]{enumitem}

\usepackage{hyperref}
\hypersetup{
colorlinks=true,
linkcolor=blue,
filecolor=blue,
urlcolor=blue,
citecolor=blue,
}

\usepackage[utf8]{inputenc}
\usepackage[T1]{fontenc}

\newtheorem{theorem}{Theorem}[section]
\newtheorem{lemma}[theorem]{Lemma}

\newtheorem{remark}[theorem]{Remark}

\setlist[enumerate]{itemsep=3pt}
\setlist[itemize]{itemsep=3pt}

\def \T  {\mathbb{T}} 
\def \R {\mathbb{R}}  
\def \N {\mathbb{N}}  

\def \p {\partial}

\def \ep {\epsilon}
\def \om {\omega}

\numberwithin{equation}{section}

\begin{document}

\title[Inviscid damping  in Sobolev spaces]{Long time inviscid damping near Couette in Sobolev spaces}

\author{Dengjun Guo}

\address{Academy of Mathematics and Systems Science, Chinese Academy of Sciences, Beijing, China}

\email{djguo@amss.ac.cn}

\author{Xiaoyutao Luo}

\address{Academy of Mathematics and Systems Science, Chinese Academy of Sciences, Beijing, China}

\email{xiaoyutao.luo@amss.ac.cn}


\date{\today}

\begin{abstract}
	We give an elementary proof of long time inviscid damping  for Sobolev perturbations near the Couette flow $(y,0)$ for the 2D Euler equations on $\T \times \R$. For any $s>1$ and any initial vorticity perturbation of size $O(\ep)$ in $H^s$, we obtain velocity damping estimates  up to a time scale   $ t = O(\ep^{-\delta_s} )$, where $\delta_s=1/3$ when $s\to 1+$ and $\delta_s=1/2$ for $s>2$. 
\end{abstract}

\maketitle

\section{Introduction}\label{sec:intro}

In this note, we consider the 2D Euler equation on $\T \times \R$ near the Couette flow $(y,0)$. The vorticity perturbation $\om : \T \times \R \to \R$ satisfies 
\begin{equation}\label{eq:eu_eq}
\begin{cases}
	\p_t \om + y\p_x \om + u \cdot \nabla \om = 0 &\\
	u = \nabla^\perp \psi, \quad \Delta \psi = \om  &\\
	\om|_{t= 0 } = \om_{in} &
\end{cases}	
\end{equation}
where    $ \nabla^\perp = (- \p_y , \p_x)$ and $\om_{in}$ is the initial data.

We are concerned with the long-time behavior of solutions to this system, specifically the phenomenon of inviscid damping~\cite{Kelvin1887,Rayleigh1879,Orr1907}--the decay of velocity perturbations in the absence of viscosity.

For the Couette flow, the linearized dynamics are well understood: phase mixing transfers energy to small scales in the vorticity, leading to algebraic decay of the velocity.

 At the nonlinear level, the situation is subtler. The breakthrough work of Bedrossian and Masmoudi \cite{MR3415068} established nonlinear inviscid damping for small perturbations in Gevrey class regularity. Important generalizations to more complex shear flows were obtained in subsequent works \cite{MR4740211,MR4628607}.

In the studies \cite{MR3415068,MR4076093,MR4740211,MR4628607}, inviscid damping follows from asymptotic stability, which requires Gevrey-2  regularity. This assumption appears to be sharp-- below Gevrey-2, instability can occur~\cite{MR4630602}.

 For Sobolev regularity, global inviscid damping is believed to fail. In particular, there exist steady states~\cite{MR2796139} and traveling waves~\cite{MR4595614} arbitrarily close to Couette flow in $H^{3/2-}$, which prohibits the global damping in such a regime.

\subsection{Main result}

The purpose of this note is to show that although global-in-time inviscid damping near Couette does not hold for Sobolev perturbations in $H^{3/2-}$, a \emph{finite but long-time} version of damping remains valid.

\begin{theorem}\label{thm:main}
For $s >1 $ and $\delta>0$ there exist $c_s>0$ and $\ep_0>0$ with the following. For any $0< \ep \leq \ep_0$ and any zero-mean initial data $\om_{in} \in H^s(\T \times \R)$ such that 
\begin{equation}\label{eq thm main 1}
\| \om_{in}\|_{H^{s} }  \leq \ep \quad \text{and} \quad \| y \om_{in}\|_{L^{2} }  \leq \ep ,
\end{equation}
define the lifespan  
\begin{equation}\label{eq thm main 2}
T_\ep: = c_s \ep^{-\delta_s }  \quad \text{with}  \quad  \delta_s = 
    \begin{cases}
  \frac{1}{4-s}  \quad\text{if $1<s<2 $  } &\\
\frac{1}{2+\delta}    \quad\text{ if $s=2 $  } &\\
\frac12 \quad\text{if $s > 2 $  }. & 
\end{cases}
\end{equation} 

 Then the (unique) solution to \eqref{eq:eu_eq} satisfies the following.

\begin{itemize}
		
	\item (Regularity):
	The re-normalized vorticity $W( t, X,Y): = \omega(t,X + Yt, Y)$ satisfies
	\begin{eqnarray}\label{eq:reg intro}
		\| W \|_{L^\infty_t ([0,T_\ep]; H^s)} \leq 3 \ep .
	\end{eqnarray}

	\item (Invisicid damping): The velocity field $u = (u^x, u^y )$ satisfies
	\begin{equation}\label{eq:damping intro}
		\begin{aligned}
			\|P_{\neq 0} u^x (t)\|_{L^2} & \lesssim \ep \langle t \rangle^{-1} \\
			\|  u^y (t) \|_{L^2} & \lesssim \ep \langle t \rangle^{- \min \{ 2,s\}}  
		\end{aligned}
		\quad \text{for all $t \in [0, T_\ep] $ }
	\end{equation}
	where $P_{\neq 0}f : = f -\int_{\T } f(x,y) \, dx $ denotes the projection of removing the $x$-mean.

\end{itemize}
 
\end{theorem}

\begin{remark}
\hfill
\begin{enumerate}
    \item Our nonlinear estimates remain valid down to $H^{1+}$ matching the well-posedness threshold of 2D Euler.  There is no contradiction with steady states\footnote{Note that the examples in \cite{MR2796139} are on the periodic channel $\T \times [-1,1]$, where the stationary structure appears near $y=0$. Nevertheless, the result in this paper is likely to hold on $\T \times [-1 ,1]$  when restricting away from the boundary.   }~\cite{MR2796139} or traveling waves~\cite{MR4595614} constructions, since the initial velocities in those examples are already smaller than the decay rate predicted by \eqref{eq:damping intro} at $t=T_\ep$.

    \item While for  $s> 2$ the decaying rates match with those of the linearized problem, it is not clear to us what would be the optimal life span $T_\ep =\ep^{-\delta_s}$.

    \item The estimates are most effective for initial data of the form $\omega_{in} = \ep \omega_0$ for some fixed profile $ \omega_0$. In the limit $\ep \to 0+$, one sees that initially $ u  \sim  \omega  \sim \ep $ and  $ u $ exhibits the algebraic decay over the time interval $[0,T_\ep]$.

    \item The damping estimate \eqref{eq:damping intro} follows from the regularity \eqref{eq:reg intro}. Whether nonlinear inviscid damping can persist globally without asymptotic stability remains open, see~\cite{MR4302767} for results in the linear case. 

    \item The instability construction in \cite{MR4630602} also satisfies our assumptions and estimates, with an eventual breakdown of \eqref{eq:reg intro} occurring long after $t=T_\ep$. A definite counterexample to inviscid damping for low Gevrey regularity is also open.
\end{enumerate}

\end{remark}

\subsection{Discussion}
The literature on stability/inviscid damping is vast; see, for example, the works  \cite{MR3710645,MR3772399,MR3987441,MR4400903} on general shears/vortices as well as \cite{MR3974608,MR4121130,MR4412070,MR4848788,MR4951440,2510.16378}  on related viscous problems and boundary effects.

In contrast, the study of nonlinear inviscid damping of 2D Euler equations in the Sobolev regime remains largely open.   The present work demonstrates that the inviscid damping remains effective on a long, but finite, time interval, which is much shorter than the nonlinear resonance breakdown time found in recent instability constructions~\cite{MR4630602}. 

Our proof is short and elementary. We work in simplified shear-back coordinates of Bedrossian-Masmoudi~\cite{MR3415068,MR4076093,MR4740211,MR4628607} and exploit the fact that the principle of ``\emph{decay costs regularity}'' can, in some sense, be reversed. In these coordinates, the new velocity field remains one derivative smoother than the re-normalized vorticity, at the cost of growing time factors.

 It is desired to find the optimal exponent $T_\ep\sim  \ep^{-\delta_s}$. However, whether the lifespan obtained here can be extended, with or without additional assumptions, is unclear to us. Another important open question is whether nonlinear inviscid damping persists without the regularity of re-normalized vorticity \eqref{eq:reg intro}. See \cite{MR4302767} for the study of a linear case.

\section{The proof}

We prove Theorem \ref{thm:main} in this section.

\subsection{The setup}

Introduce the shear back coordinates $(x,y) \mapsto (X,Y) \in \T\times \R$,
\begin{equation}
	\begin{cases}
		 x= X +t  Y ,&\\ 
		  y = Y
	\end{cases}
\end{equation} 
The smooth bijection $(x,y) \mapsto (X,Y)$ is isomorphic on $L^p$, $1\leq p \leq \infty$.

Define the shear-back quantities in capital letters. 
\begin{itemize} \item The stream function:
	\begin{equation}
		\Psi(t, X,Y) = \psi(t,X +t  Y , Y ) .
	\end{equation}
	
	\item The vorticity
	\begin{equation}
		W(t, X,Y) = \om(t, X +t  Y , Y ) .
	\end{equation}
\end{itemize}

Note the chain rules $
	\p_t   = \p_t  +  Y  \p_x$,   $
	\p_X       = \p_x $, and 
	$\p_Y     = t   \p_x  + \p_y$.  The equation \eqref{eq:eu_eq} is equivalent to 
\begin{equation}
	\p_t W +  u \cdot \nabla \om = 0 .
\end{equation}

Introduce the re-normalized velocity:
\begin{eqnarray}
	U : = (U^X, U^Y) = (u^x -t u^y, u^y) .
\end{eqnarray}
We check that $U = \nabla^\perp_{X,Y} \Psi$; indeed,
\begin{eqnarray}
 	- \p_Y \Psi = -(t   \p_x  + \p_y) \psi = -t u^y + u^x = U^x \quad \text{and} \quad \p_X \Psi = \p_x   \psi =  u^y .
\end{eqnarray} 
The transport term $u \cdot \nabla \om$ is equivalent to 
\begin{equation}
u\cdot (\p_X W, \p_Y W - t \p_X W) = ( u^x -tu^y )\p_X W + u^y \p_Y W = U \cdot \nabla_{X,Y} W .
\end{equation}

To summarize, we have reformulated \eqref{eq:eu_eq} in re-normalized coordinates $X,Y \in \T \times \R$
\begin{equation}\label{eq:eu_XY}
	\begin{cases}
		\p_t W +    U \cdot \nabla  W = 0 &\\
		U = \nabla^\perp_{X,Y} \Psi, \quad \Delta_t \Psi = W  &\\
	W |_{t= 0 } = \om_{in}(X,Y) &
	\end{cases}	
\end{equation}
where $\Delta_t : = (\p_Y - t \p_X)^2 + \p_X^2 $. As we shall see, this system can be handled by the classical Kato-Ponce inequality~\cite{MR0951744,MR3914540}.

\subsection{A priori  estimates}
To control the evolution, we use the following estimates that allow us to trade $t$ factors for up to one derivative.

Recall the Sobolev norms $\| \cdot \|_{H^s} = \|J^s \cdot \|_{L^2} $ and  $\| \cdot \|_{\dot H^s} = \|\Lambda^s \cdot \|_{L^2} $ where  the multipliers $\widehat{J}^s: = \langle \xi\rangle^s $ and $\widehat{\Lambda}^s: =  |\xi |^s $ for any $s\in \R$.  Here and in what follows, $\langle \cdot \rangle =\sqrt{1+|\cdot|^2}$ is the Japanese bracket.
 
\begin{lemma}\label{lemma:U est}
Suppose the pair $(U, W) $ satisfies  $U=\nabla^\perp \Delta^{-1}_t   W$ with $\int_{\T \times \R} W = 0 $ and $   Y W \in L^2 $.

Then for any $s_1,s_2\in \R$ such that $-1 \leq   s_1- s_2 \leq 1$, there hold
	\begin{equation}\label{eq:lemma_1}
    \begin{aligned} 
		\|   P_{\neq 0} U \|_{ \dot H^{s_1 }  }    &\lesssim   \langle t \rangle^{1+s_1- s_2 } 	\|  W \|_{\dot  H^{s_2 }  }  
    \end{aligned} 
	\end{equation}
 and 
 \begin{equation}\label{eq:lemma_2}
\|     U \|_{ L^2  }   \lesssim    	\| \langle  Y\rangle W \|_{  L^2  }  .
\end{equation}

\end{lemma}
\begin{proof} 

\noindent
\textbf{Proof of \eqref{eq:lemma_1}}

Since $\Delta_t$ is constant-coefficient, we note the commutation $
	[\Delta_t , \Lambda^s ] = 	[\Delta_t , J^s ]  = 	   0
$.
Thanks to the projection $P_{\neq 0}$, by Plancherel $\|  \Delta^{-1}_t P_{\neq 0} W \|_{L^2  }  \lesssim     \|   W \|_{L^2  }$.

Therefore, by interpolation, it suffices to prove 
\begin{align} 
    \| \nabla^2 \Delta^{-1}_t  W \|_{L^2  } & \lesssim   \langle t \rangle^2  \|   W \|_{L^2  }  .\label{eq:lemma_aux2}
\end{align}

We now demonstrate \eqref{eq:lemma_aux2}. Recall $w(t,x,y)=W(x-ty,y)$ and $\Delta_t^{-1}W(X,Y)=\Delta^{-1}w(t,X+tY,Y)$.
By standard singular operator bounds $ \| \nabla_{x,y}^2 \Delta^{-1}   \|_{L^2_{x,y} \to L^2_{x,y}} \lesssim 1  $, and note that for any second order derivative of $XY$, $
   D_{XY}^\alpha =  t^2 C_{\alpha ,1} \p_x^2 + t^1 C_{\alpha ,2} \p_x\p_y + C_{\alpha ,3} \p_y^2$,
we obtain \eqref{eq:lemma_aux2}.

\noindent
\textbf{Proof of \eqref{eq:lemma_2}}

By writing $U =P_{\neq 0}     U +   P_{=0}U $ and the previous argument, it suffices to only consider the $x$-mean part. Since there is no $x$-variable, we have  
$$ 
\| P_{=0}U  \|_{L^2} = \| \p_y \Delta^{-1}   P_{=0}W  \|_{L^2}  =   \| |\p_Y|^{-1}  P_{=0} W \|_{L^2_Y}  .
$$
Note that $|\p_Y|^{-1}   P_{=0}W$ is well-defined since $\int_{\T \times \R} W (x,y) \,dx\, dy = \int_{\R} P_{=0}W(y) \, dy  = 0 $. Moreover, $ Y P_{=0}W \in L^2(\R)$ and $\widehat{P_{=0} W}(0)=0$, so Hardy's inequality gives
$$
\| |\p_Y|^{-1}   P_{=0} W \|_{L^2_Y}  \leq \|  \p_\xi  \widehat{P_{=0} W} \|_{L^2_\xi}\lesssim \|\p_{\xi} \widehat{W}\|_{L^2} \lesssim \|  Y  W \|_{L^2} .
$$

\end{proof}

The main step of the proof is the following a priori estimates.

\begin{lemma}\label{lemma:Hs a priori}
Let $s>1$. The system \eqref{eq:eu_XY} satisfies the a priori estimates:
\begin{equation}
\frac{	d \| W(t) \|_{H^s} }{dt} \lesssim \| \nabla U \|_{L^\infty }  \|   W \|_{H^s} +
\begin{cases}
\|  J^{ 2} U \|_{ L^2 }  \|   W \|_{H^s}    & \quad \text{1<s<2}\\
\|  J^{  2+\delta} U \|_{ L^2 }  \|   W \|_{H^s}  &\quad \text{ s=2} \\
\|  J^{ s} U \|_{ L^2 }  \|   W \|_{H^s}  &\quad \text{ s>2}.
\end{cases}   
\end{equation}

\end{lemma}

\begin{proof}
Since $U$ is divergence-free, when $\frac1p+\frac1q=\frac12$, by Kato-Ponce  we have
\begin{equation}\label{eq:aux katoponce}
\frac{	d \| W(t) \|_{H^s}^2 }{dt} \lesssim  \| \nabla U \|_{L^\infty }  \| J^{ s} W\|_{L^2}^2  + \| J^s  U \|_{L^p}  \| \nabla W \|_{L^q}   \| J^{ s} W\|_{L^2} .
\end{equation}
Here, the choice of $p,q$ depends on $s$, and we separate into three cases.

\noindent 
\textbf{Case 1: $1<s < 2$}

In this case we take $2< q = \frac{2}{2-s} <\infty $  and $2 < p = \frac{2}{s-1} < \infty $, and Sobolev embedding gives  
\begin{equation}
\begin{aligned}
\| J^s  U \|_{L^p}  \| \nabla W \|_{L^q}  \lesssim    \| J^2  U \|_{L^2}      \| J^{ s} W\|_{L^2}   .
\end{aligned}
\end{equation} 

\noindent 
\textbf{Case 2: $ s =2$}

For any $\delta>0$, we take $  p >2  $ close to $2$ so that $H^{\delta} (\T \times \R )\hookrightarrow  L^{p}(\T \times \R )$. So in \eqref{eq:aux katoponce}, there holds $ \| J^s  U \|_{L^p} \lesssim  \| J^{2+\delta}  U \|_{L^2}  $     for any $\delta>0$ and    $\| \nabla W \|_{L^q} \lesssim \| J^s W \|_{L^2}$, which concludes the proof.

\noindent 
\textbf{Case 3: $s >  2$}
 
 In this case, take $p=2$ and $q=\infty$. The estimate follows from  the embedding $  H^s(\T \times \R ) \hookrightarrow  L^\infty(\T \times \R )$ for $s>1$.

\end{proof}

We also need to track the second condition in \eqref{eq thm main 1}.

\begin{lemma}\label{lemma:a priori total}
Let $s>1$ and define the norm
$$
\|    f\|_{ \bar{H}^s } = \|   f\|_{  {H}^s } + \|   Y f   \|_{ L^2 }.
$$

If $\int_{\T \times \R} \om_{in} = 0 $ and $   y \om_{in} \in L^2 $, then the system \eqref{eq:eu_XY} admits the following a priori estimates:
\begin{equation}\label{eq prop:apriori_yw}
\frac{	d \|   W (t) \|_{ \bar{H}^s } }{dt} \lesssim \langle  t\rangle^{\beta_s} \|W \|_{\bar{H}^s}^2  ,
\end{equation} 
where the exponent $  \beta_s =3-s $ if $1<s<2 $, $  \beta_s =1+\delta $ for arbitrary $\delta>0$  if $s=2 $, and $  \beta_s =1 $ if $s > 2 $.

\end{lemma}
\begin{proof}
To show \eqref{eq prop:apriori_yw} we first compute the evolution of $\|  Y   W (t) \|_{L^2}$. Multiply by $  Y^2 W$ and integrate by parts,
   \begin{eqnarray}\label{eq aux apriori 1}
 \frac{	d \|  Y   W (t) \|_{L^2}^2 }{dt} \leq \| U^Y \|_{L^\infty } \|   Y   W  \|_{L^2}\|      W  \|_{L^2} .
   \end{eqnarray}
Since $\| U^Y \|_{L^\infty } \lesssim \| P_{\neq 0} U \|_{H^s } $ for $s>1$, by \eqref{eq:lemma_1} we have
   \begin{align}\label{eq aux apriori 1b}
 \frac{	d \|   Y   W (t) \|_{L^2}  }{dt}   
 & \leq \langle  t\rangle     \|      W  \|_{H^s}^2 .
   \end{align}
In the rest, we focus on the $H^s$ part, bounding the estimates given by Lemma \ref{lemma:Hs a priori}.

\noindent 
\textbf{Case 1: $1<s \leq 2$}

Since $ 2-s<1<s$, we use the interpolation  
\begin{equation}\label{eq aux apriori 2}
  \| \nabla U \|_{L^\infty } \lesssim \|  U \|_{\dot H^{s+1} }^\frac{1}{2}  \|  U \|_{\dot H^{3-s} }^\frac{1}{2}   .
\end{equation}
Separating zero and non-zero $x$-modes, by Lemma \ref{lemma:U est} we have
\begin{align}
\|  U \|_{\dot H^{s+1} }& \leq \|  P_{ =0 }  U \|_{\dot H^{s+1} }  + \| P_{ \neq 0 }  U \|_{ \dot H^{s+1} }   \nonumber\\ 
& \lesssim \|P_{ =0 } W\|_{\dot H^{s } } + \langle t\rangle^2 \| W \|_{ \dot  H^{s } }   \lesssim    \langle t\rangle^2 \| W \|_{   H^{s } } , \label{eq aux apriori 3}
\end{align}
and similarly (since $-1 \le 3-2s \le 1$)
\begin{align}
\|  U \|_{\dot H^{3-s} }  & \lesssim \|P_{ =0 } W\|_{\dot H^{2-s } } + \langle t\rangle^{4-2s} \| W \|_{  H^{s } } \nonumber\\
& \lesssim  \langle t\rangle^{4-2s} \| W \|_{  H^{s } } .\label{eq aux apriori 4}
\end{align}

It follows from \eqref{eq aux apriori 2}, \eqref{eq aux apriori 3} and \eqref{eq aux apriori 4} that
 \begin{equation}\label{eq aux apriori 5}
\begin{aligned}
\| \nabla U \|_{L^\infty }& \lesssim      \langle  t\rangle^{3- s}      \| W(t) \|_{H^s}.
\end{aligned}
\end{equation} 
 Next, we consider the $J^2U$ term ($J^{2+}U$ when $s=2$ is similar).
\begin{equation}\label{eq aux apriori 6}
\| J^2  U \|_{L^2} \lesssim \|   U \|_{L^2} + \|    U \|_{\dot H^2} \lesssim \| \langle  Y\rangle W \|_{  L^2  }   + \langle  t\rangle^{3- s}      \| W(t) \|_{H^s}.
\end{equation}

Combining \eqref{eq aux apriori 1b}, \eqref{eq aux apriori 5} and \eqref{eq aux apriori 6} with Lemma \ref{lemma:Hs a priori}, we have:
\begin{equation*}
\frac{	d \| W(t) \|_{\bar H^s} }{dt}    
\begin{cases}
\lesssim  \langle  t\rangle^{3- s}      \| W(t) \|_{\bar H^s}^2 \quad & \text{if $1<s<2$}\\
\lesssim_\delta  \langle  t\rangle^{1+\delta}      \| W(t) \|_{\bar H^s}^2 \quad & \text{if $ s =2$}.
\end{cases}
\end{equation*}

\noindent 
\textbf{Case 2: $ s >2$}

In this case, we also split into zero and non-zero $x$-modes and then use Sobolev embedding to obtain
\begin{align}
  \| \nabla U \|_{L^\infty } &\lesssim \| \nabla U \|_{  H^{s -1} }      \lesssim \| \nabla P_{= 0 } U \|_{  H^{s -1} } + \langle t \rangle  \| W \|_{ H^{ s} } \nonumber \\
  & \lesssim \|  P_{= 0 } W \|_{  H^{s -1} } + \langle t \rangle  \| W \|_{H^{ s} } \lesssim   \langle t \rangle  \| W \|_{  H^{ s} } .\label{eq aux apriori 7}
\end{align}
and 
\begin{equation}\label{eq aux apriori 8}
  \| J^s U \|_{L^2 } \lesssim \| U \|_{L^2 } +  \| \Lambda^s U \|_{L^2 } \lesssim \| \langle  Y\rangle W \|_{  L^2  }    + \langle t \rangle  \| W \|_{ H^{ s} }.
\end{equation}

Similarly to the Case 1, by \eqref{eq aux apriori 7}, \eqref{eq aux apriori 8} and Lemma \ref{lemma:Hs a priori} and we obtain
\begin{equation*}
\begin{aligned}
\frac{	d \| W(t) \|_{\bar H^s} }{dt}     
& \lesssim     \langle  t\rangle      \| W(t) \|_{\bar H^s}^2.
\end{aligned}
\end{equation*}

\end{proof}

\subsection{Long-time regularity of $W$}
With Lemma \ref{lemma:a priori total} established, we obtain the bound 
$$
\| W(t) \|_{\bar H^s} \leq \frac{1}{ (2\ep)^{-1} -C_s t^{1+ \beta_s} } .  
$$ 

In particular, if $c_s>0$ is sufficiently small and $T_\ep = c_s \ep^{ - \delta_s }$, then  there holds
\begin{equation}
\| W \|_{L^\infty ([0,T_\ep]; \bar  H^s)} \leq 3 \ep .
\end{equation}

\subsection{Long-time damping estimates}

Once we have the $H^s$ regularity of $W$, it is standard to derive inviscid damping estimates. See, for instance \cite{MR2796139}. Note that the lemma below is ``frozen-time''   and does not use the equation \eqref{eq:eu_XY}.

\begin{lemma}\label{lemma:damping}
For any $s \geq 0 $ and any $\om(x,y): = W(x -yt, y)$, there holds
\begin{equation}\label{eq:damping}
	\| P_{ \neq 0 } \om \|_{\dot H^{-s}  }   \lesssim  
		\langle t \rangle^{-s}  \|     W \|_{L^2_X H^s_Y  }   
\end{equation}		 
\end{lemma}

We conclude that the decay estimate \eqref{eq:damping intro} follows from Lemma \ref{lemma:damping}; for $u^x$,
\begin{align*}
\| P_{\neq 0 } u^x (t)\|_{L^2}   \lesssim   \|   P_{\neq 0 } \om \|_{\dot H^{-1}}  
    \lesssim \langle t\rangle^{-1} \|W \|_{L^\infty_t([0,T_\ep]; H^1)}
\end{align*}
while for $u^y$ (applying Lemma \ref{lemma:damping} for $\p_x \om $ and $\p_X W$),
\begin{align*}
\|   u^y (t)\|_{L^2}    = \| P_{\neq 0 } \Delta^{-1}_t ( \p_X W) \|_{L^2}   \lesssim  \langle t \rangle^{- \min \{ 2,s\}}   \|W \|_{L^\infty_t([0,T_\ep];H^s) } 
\end{align*}
for any $s>1$ since the extra $\p_X$ can be absorbed in $H^s$. We have thus proven Theorem \ref{thm:main}

Finally, we prove Lemma \ref{lemma:damping}. 

\begin{proof}[Proof of Lemma \ref{lemma:damping}]
We demonstrate  \eqref{eq:damping} for integers $s$; the general case follows from separating the high and low $Y$-frequencies. It also suffices to consider $t\geq 1$.

\noindent 
\textbf{Case 1: $ s =1$}

By standard duality, it suffices to bound
\begin{equation}
	\| P_{ \neq 0 } \om \|_{\dot H^{-1}  } = \sup_{ 
    |\varphi \|_{\dot H^1} \leq 1 ; P_{=0} \varphi  = 0} \int \om \varphi \, dx dy.
\end{equation}
So we fix such a smooth test function $\varphi \in C^\infty_c(\T \times \R)$ with zero $x$-mean and consider
\begin{equation}\label{eq:om_int}
	\int \om \varphi \, dx dy =\int W  (X, Y )\varphi( X+ t Y, Y) \,dX dY.
\end{equation}
Because $\varphi$ has zero x-mean and periodic in $x$, there exists a   unique zero x-mean function $ \phi= \p_x^{-1} \varphi$.  
Then we use the identity
\begin{equation}\label{eq:om_int 2}
	\varphi(\cdot) = \frac{1}{t}\p_Y \left( \phi(\cdot) \right) - \frac{1}{t}\p_y \phi(\cdot)
\end{equation}
where the argument $(\cdot) = (X+ t Y, Y )$ is omitted.
Inserting this into \eqref{eq:om_int} gives
\begin{align}
	\int \om \varphi \, dx dy & =-\frac{1}{t} \int \p_Y W  \phi(\cdot) \,dX dY \nonumber  -\frac{1}{t} \int  W  \p_y \phi(\cdot) \,dX dY \nonumber \\
	& \leq  t^{-1} \| \p_Y W \|_{L^2 }  \| \phi \|_{L^2}  +   t^{-1}  \|    W \|_{L^2 }  \| \p_y \phi \|_{L^2 } . \label{eq:aux10} 
\end{align}
Note that $\left| \widehat{\phi}(k,\xi) \right|=\left| \frac{\widehat{\varphi}(k,\xi)}{k} \right|$, we have $\| \phi \|_{H^1} \lesssim  \|  \varphi \|_{\dot H^1}  $, which gives
\begin{align}
	\Big|\int \om \varphi \, dx dy \Big| & \lesssim \frac{1}{t}  \|   W \|_{L^2_X  H^1_Y}  \|   \varphi \|_{\dot H^1 }  \nonumber .
\end{align}

\noindent 
\textbf{Case 2: $ \N \ni s \geq 2$}

 The proof is nearly the same. We just need to integrate by parts repeatedly for the two terms on the right-hand side of \eqref{eq:aux10} using \eqref{eq:om_int 2}.

\end{proof}

\subsection*{Acknowledgment}

We thank Siming He for helpful comments and references. XL is grateful to Qingtang Su for introducing him the topic.

\appendix

\bibliographystyle{alpha}
\bibliography{fluid_illposed.bib}

@article {MR3415068,
    AUTHOR = {Bedrossian, Jacob and Masmoudi, Nader},
     TITLE = {Inviscid damping and the asymptotic stability of planar shear
              flows in the 2{D} {E}uler equations},
   JOURNAL = {Publ. Math. Inst. Hautes \'Etudes Sci.},
  FJOURNAL = {Publications Math\'ematiques. Institut de Hautes \'Etudes
              Scientifiques},
    VOLUME = {122},
      YEAR = {2015},
     PAGES = {195--300},
      ISSN = {0073-8301,1618-1913},
   MRCLASS = {35Q31 (35B35 35Q35)},
  MRNUMBER = {3415068},
MRREVIEWER = {Matthew\ Paddick},
       DOI = {10.1007/s10240-015-0070-4},
       URL = {https://doi.org/10.1007/s10240-015-0070-4},
}

@article {MR4951440,
    AUTHOR = {Li, Hui and Masmoudi, Nader and Zhao, Weiren},
     TITLE = {Asymptotic stability of two-dimensional {C}ouette flow in a
              viscous fluid},
   JOURNAL = {Arch. Ration. Mech. Anal.},
  FJOURNAL = {Archive for Rational Mechanics and Analysis},
    VOLUME = {249},
      YEAR = {2025},
    NUMBER = {5},
     PAGES = {Paper No. 56, 113},
      ISSN = {0003-9527,1432-0673},
   MRCLASS = {76D05 (35B35 35B40 76E05)},
  MRNUMBER = {4951440},
       DOI = {10.1007/s00205-025-02129-5},
       URL = {https://doi.org/10.1007/s00205-025-02129-5},
}

@article {MR3987441,
    AUTHOR = {Bedrossian, Jacob and Coti Zelati, Michele and Vicol, Vlad},
     TITLE = {Vortex axisymmetrization, inviscid damping, and vorticity
              depletion in the linearized 2{D} {E}uler equations},
   JOURNAL = {Ann. PDE},
  FJOURNAL = {Annals of PDE. Journal Dedicated to the Analysis of Problems
              from Physical Sciences},
    VOLUME = {5},
      YEAR = {2019},
    NUMBER = {1},
     PAGES = {Paper No. 4, 192},
      ISSN = {2524-5317,2199-2576},
   MRCLASS = {76B47 (35B40 35P25 35Q31 76B03)},
  MRNUMBER = {3987441},
MRREVIEWER = {Xinyu\ He},
       DOI = {10.1007/s40818-019-0061-8},
       URL = {https://doi.org/10.1007/s40818-019-0061-8},
}

@article {MR3974608,
    AUTHOR = {Bedrossian, Jacob and Germain, Pierre and Masmoudi, Nader},
     TITLE = {Stability of the {C}ouette flow at high {R}eynolds numbers in
              two dimensions and three dimensions},
   JOURNAL = {Bull. Amer. Math. Soc. (N.S.)},
  FJOURNAL = {American Mathematical Society. Bulletin. New Series},
    VOLUME = {56},
      YEAR = {2019},
    NUMBER = {3},
     PAGES = {373--414},
      ISSN = {0273-0979,1088-9485},
   MRCLASS = {76E05 (35B25 35B34 35B35 35Q30 76D05 76E30)},
  MRNUMBER = {3974608},
MRREVIEWER = {Georgios\ C.\ Georgiou},
       DOI = {10.1090/bull/1649},
       URL = {https://doi.org/10.1090/bull/1649},
}

@article {MR4628607,
    AUTHOR = {Ionescu, Alexandru D. and Jia, Hao},
     TITLE = {Non-linear inviscid damping near monotonic shear flows},
   JOURNAL = {Acta Math.},
  FJOURNAL = {Acta Mathematica},
    VOLUME = {230},
      YEAR = {2023},
    NUMBER = {2},
     PAGES = {321--399},
      ISSN = {0001-5962,1871-2509},
   MRCLASS = {76E30 (35Q35 76B03)},
  MRNUMBER = {4628607},
MRREVIEWER = {Yanguang\ (Charles)\ Li},
       DOI = {10.4310/acta.2023.v230.n2.a2},
       URL = {https://doi.org/10.4310/acta.2023.v230.n2.a2},
}

@article {MR4412070,
    AUTHOR = {Masmoudi, Nader and Zhao, Weiren},
     TITLE = {Stability threshold of two-dimensional {C}ouette flow in
              {S}obolev spaces},
   JOURNAL = {Ann. Inst. H. Poincar\'e{} C Anal. Non Lin\'eaire},
  FJOURNAL = {Annales de l'Institut Henri Poincar\'e{} C. Analyse Non
              Lin\'eaire},
    VOLUME = {39},
      YEAR = {2022},
    NUMBER = {2},
     PAGES = {245--325},
      ISSN = {0294-1449,1873-1430},
   MRCLASS = {35B35 (35B40 35Q30)},
  MRNUMBER = {4412070},
MRREVIEWER = {Zhong\ Wang},
       DOI = {10.4171/aihpc/8},
       URL = {https://doi.org/10.4171/aihpc/8},
}

@article {MR4400903,
    AUTHOR = {Ionescu, Alexandru D. and Jia, Hao},
     TITLE = {Axi-symmetrization near point vortex solutions for the 2{D}
              {E}uler equation},
   JOURNAL = {Comm. Pure Appl. Math.},
  FJOURNAL = {Communications on Pure and Applied Mathematics},
    VOLUME = {75},
      YEAR = {2022},
    NUMBER = {4},
     PAGES = {818--891},
      ISSN = {0010-3640,1097-0312},
   MRCLASS = {76D17 (76E30)},
  MRNUMBER = {4400903},
MRREVIEWER = {Takashi\ Sakajo},
       DOI = {10.1002/cpa.21974},
       URL = {https://doi.org/10.1002/cpa.21974},
}

@article {MR4740211,
    AUTHOR = {Masmoudi, Nader and Zhao, Weiren},
     TITLE = {Nonlinear inviscid damping for a class of monotone shear flows
              in a finite channel},
   JOURNAL = {Ann. of Math. (2)},
  FJOURNAL = {Annals of Mathematics. Second Series},
    VOLUME = {199},
      YEAR = {2024},
    NUMBER = {3},
     PAGES = {1093--1175},
      ISSN = {0003-486X,1939-8980},
   MRCLASS = {35Q31 (76E05)},
  MRNUMBER = {4740211},
       DOI = {10.4007/annals.2024.199.3.3},
       URL = {https://doi.org/10.4007/annals.2024.199.3.3},
}

@article {MR4630602,
    AUTHOR = {Deng, Yu and Masmoudi, Nader},
     TITLE = {Long-time instability of the {C}ouette flow in low {G}evrey
              spaces},
   JOURNAL = {Comm. Pure Appl. Math.},
  FJOURNAL = {Communications on Pure and Applied Mathematics},
    VOLUME = {76},
      YEAR = {2023},
    NUMBER = {10},
     PAGES = {2804--2887},
      ISSN = {0010-3640,1097-0312},
   MRCLASS = {76E30},
  MRNUMBER = {4630602},
       DOI = {10.1002/cpa.22092},
       URL = {https://doi.org/10.1002/cpa.22092},
}

@article {MR4595614,
    AUTHOR = {Castro, \'Angel and Lear, Daniel},
     TITLE = {Traveling waves near {C}ouette flow for the 2{D} {E}uler
              equation},
   JOURNAL = {Comm. Math. Phys.},
  FJOURNAL = {Communications in Mathematical Physics},
    VOLUME = {400},
      YEAR = {2023},
    NUMBER = {3},
     PAGES = {2005--2079},
      ISSN = {0010-3616,1432-0916},
   MRCLASS = {76B03 (35C07 35Q31 76E05)},
  MRNUMBER = {4595614},
MRREVIEWER = {Delyan\ Zhelyazov},
       DOI = {10.1007/s00220-023-04636-6},
       URL = {https://doi.org/10.1007/s00220-023-04636-6},
}

@article {MR4076093,
    AUTHOR = {Ionescu, Alexandru D. and Jia, Hao},
     TITLE = {Inviscid damping near the {C}ouette flow in a channel},
   JOURNAL = {Comm. Math. Phys.},
  FJOURNAL = {Communications in Mathematical Physics},
    VOLUME = {374},
      YEAR = {2020},
    NUMBER = {3},
     PAGES = {2015--2096},
      ISSN = {0010-3616,1432-0916},
   MRCLASS = {35Q31 (35B35 76B99)},
  MRNUMBER = {4076093},
MRREVIEWER = {Ionu\c t\ Munteanu},
       DOI = {10.1007/s00220-019-03550-0},
       URL = {https://doi.org/10.1007/s00220-019-03550-0},
}

@article {MR3710645,
    AUTHOR = {Zillinger, Christian},
     TITLE = {Linear inviscid damping for monotone shear flows},
   JOURNAL = {Trans. Amer. Math. Soc.},
  FJOURNAL = {Transactions of the American Mathematical Society},
    VOLUME = {369},
      YEAR = {2017},
    NUMBER = {12},
     PAGES = {8799--8855},
      ISSN = {0002-9947,1088-6850},
   MRCLASS = {76E05 (35B35 35P25 35Q31 35Q35 76B03)},
  MRNUMBER = {3710645},
MRREVIEWER = {Iuliana\ Oprea},
       DOI = {10.1090/tran/6942},
       URL = {https://doi.org/10.1090/tran/6942},
}

@article {MR4121130,
    AUTHOR = {Chen, Qi and Li, Te and Wei, Dongyi and Zhang, Zhifei},
     TITLE = {Transition threshold for the 2-{D} {C}ouette flow in a finite
              channel},
   JOURNAL = {Arch. Ration. Mech. Anal.},
  FJOURNAL = {Archive for Rational Mechanics and Analysis},
    VOLUME = {238},
      YEAR = {2020},
    NUMBER = {1},
     PAGES = {125--183},
      ISSN = {0003-9527,1432-0673},
   MRCLASS = {76D05 (35Q30)},
  MRNUMBER = {4121130},
       DOI = {10.1007/s00205-020-01538-y},
       URL = {https://doi.org/10.1007/s00205-020-01538-y},
}

@article {MR3772399,
    AUTHOR = {Wei, Dongyi and Zhang, Zhifei and Zhao, Weiren},
     TITLE = {Linear inviscid damping for a class of monotone shear flow in
              {S}obolev spaces},
   JOURNAL = {Comm. Pure Appl. Math.},
  FJOURNAL = {Communications on Pure and Applied Mathematics},
    VOLUME = {71},
      YEAR = {2018},
    NUMBER = {4},
     PAGES = {617--687},
      ISSN = {0010-3640,1097-0312},
   MRCLASS = {35Q31 (76B03)},
  MRNUMBER = {3772399},
MRREVIEWER = {Xinyu\ He},
       DOI = {10.1002/cpa.21672},
       URL = {https://doi.org/10.1002/cpa.21672},
}

@article {MR2796139,
    AUTHOR = {Lin, Zhiwu and Zeng, Chongchun},
     TITLE = {Inviscid dynamical structures near {C}ouette flow},
   JOURNAL = {Arch. Ration. Mech. Anal.},
  FJOURNAL = {Archive for Rational Mechanics and Analysis},
    VOLUME = {200},
      YEAR = {2011},
    NUMBER = {3},
     PAGES = {1075--1097},
      ISSN = {0003-9527,1432-0673},
   MRCLASS = {76B03 (35L60 35Q35)},
  MRNUMBER = {2796139},
MRREVIEWER = {David\ M.\ Ambrose},
       DOI = {10.1007/s00205-010-0384-9},
       URL = {https://doi.org/10.1007/s00205-010-0384-9},
}

@article {MR4848788,
    AUTHOR = {Bedrossian, Jacob and He, Siming and Iyer, Sameer and Wang,
              Fei},
     TITLE = {Stability threshold of nearly-{C}ouette shear flows with
              {N}avier boundary conditions in 2{D}},
   JOURNAL = {Comm. Math. Phys.},
  FJOURNAL = {Communications in Mathematical Physics},
    VOLUME = {406},
      YEAR = {2025},
    NUMBER = {2},
     PAGES = {Paper No. 28, 42},
      ISSN = {0010-3616,1432-0916},
   MRCLASS = {35Q30 (35B40 76E05)},
  MRNUMBER = {4848788},
MRREVIEWER = {Zijin\ Li},
       DOI = {10.1007/s00220-024-05175-4},
       URL = {https://doi.org/10.1007/s00220-024-05175-4},
}

@article{Rayleigh1879,
author = {Rayleigh, Lord},
title = {On the Stability, or Instability, of certain Fluid Motions},
journal = {Proceedings of the London Mathematical Society},
volume = {s1-11},
number = {1},
pages = {57-72},
doi = {https://doi.org/10.1112/plms/s1-11.1.57},
url = {https://londmathsoc.onlinelibrary.wiley.com/doi/abs/10.1112/plms/s1-11.1.57},
eprint = {https://londmathsoc.onlinelibrary.wiley.com/doi/pdf/10.1112/plms/s1-11.1.57},
year = {1879}
}

@article{Kelvin1887,
author = {Kelvin, Lord},
 journal = {Philosophical Magazine},
 pages = {188--196},
 publisher = {Royal Irish Academy},
 title = {Stability of fluid motion: rectilinear motion of viscous fluid between two parallel plates},
 urldate = {2025-11-11},
 volume = {24},
NUMBER = {5},
 year = {1887}
}

@article{Orr1907,
 ISSN = {00358975},
 URL = {http://www.jstor.org/stable/20490591},
 author = {William M'F. Orr},
 journal = {Proceedings of the Royal Irish Academy. Section A: Mathematical and Physical Sciences},
 pages = {69--138},
 publisher = {Royal Irish Academy},
 title = {The Stability or Instability of the Steady Motions of a Perfect Liquid and of a Viscous Liquid. Part II: A Viscous Liquid},
 urldate = {2025-11-11},
 volume = {27},
 year = {1907}
}

@article {MR4302767,
    AUTHOR = {Deng, Yu and Zillinger, Christian},
     TITLE = {Echo chains as a linear mechanism: norm inflation, modified
              exponents and asymptotics},
   JOURNAL = {Arch. Ration. Mech. Anal.},
  FJOURNAL = {Archive for Rational Mechanics and Analysis},
    VOLUME = {242},
      YEAR = {2021},
    NUMBER = {1},
     PAGES = {643--700},
      ISSN = {0003-9527,1432-0673},
   MRCLASS = {76B03 (35Q31)},
  MRNUMBER = {4302767},
MRREVIEWER = {Francesco\ Fanelli},
       DOI = {10.1007/s00205-021-01697-6},
       URL = {https://doi.org/10.1007/s00205-021-01697-6},
}

@article {MR3914540,
    AUTHOR = {Li, Dong},
     TITLE = {On {K}ato-{P}once and fractional {L}eibniz},
   JOURNAL = {Rev. Mat. Iberoam.},
  FJOURNAL = {Revista Matem\'atica Iberoamericana},
    VOLUME = {35},
      YEAR = {2019},
    NUMBER = {1},
     PAGES = {23--100},
      ISSN = {0213-2230,2235-0616},
   MRCLASS = {35R11 (35Q35 35Q86)},
  MRNUMBER = {3914540},
MRREVIEWER = {Vincenzo\ Ambrosio},
       DOI = {10.4171/rmi/1049},
       URL = {https://doi.org/10.4171/rmi/1049},
}

@article {MR0951744,
    AUTHOR = {Kato, Tosio and Ponce, Gustavo},
     TITLE = {Commutator estimates and the {E}uler and {N}avier-{S}tokes
              equations},
   JOURNAL = {Comm. Pure Appl. Math.},
  FJOURNAL = {Communications on Pure and Applied Mathematics},
    VOLUME = {41},
      YEAR = {1988},
    NUMBER = {7},
     PAGES = {891--907},
      ISSN = {0010-3640,1097-0312},
   MRCLASS = {35Q10 (47F05 76D05)},
  MRNUMBER = {951744},
MRREVIEWER = {Josef\ Bemelmans},
       DOI = {10.1002/cpa.3160410704},
       URL = {https://doi.org/10.1002/cpa.3160410704},
}

@article{2510.16378,
      title={Stability threshold of close-to-{C}ouette shear flows with no-slip boundary conditions in 2D}, 
      author={Jacob Bedrossian and Siming He and Sameer Iyer and Linfeng Li and Fei Wang},
      year={2025},
      journal={preprint},
      eprint={2510.16378},
      archivePrefix={arXiv},
      primaryClass={math.AP},
      url={https://arxiv.org/abs/2510.16378}, 
}

\end{document}